\documentclass{amsart}
\usepackage{amssymb, amsmath, amsthm}
\usepackage[all]{xy}
\usepackage{hyperref} 
\usepackage{fullpage}
\usepackage{setspace} 
\input{macros.sty}

\theoremstyle{plain}
\newtheorem{thm}{Theorem}[section]
\newtheorem{lemma}[thm]{Lemma}
\newtheorem{cor}[thm]{Corollary}

\theoremstyle{definition}
\newtheorem{definition}[thm]{Definition}
\newtheorem{Ex}[thm]{Example}

\begin{document}
\title{Positivity of cotangent bundles}
\author{Kelly Jabbusch}
\address{Mathematisches Institut, Universit\"at zu K\"oln, Weyertal 86-90, 50931 K\"oln, Germany}
\email{kjabbusc@math.uni-koeln.de}
\date{\today}
\maketitle

\section{Introduction}
Let $X$ be a projective scheme over an algebraically closed field.  Given a vector bundle $\sE$ on $X$, we can consider various notions of positivity for $\sE$, such as ample, nef and big. As a particular example, consider a smooth projective variety $X$ and its cotangent bundle $\OM_X$.  When $\OM_X$ is ample, $X$ has some very nice properties.  For example, all subvarieties of $X$ are of general type and $X$ is algebraically hyperbolic, so in particular, $X$ does not contain rational or elliptic curves,  there do not exist non-constant maps $f: A \to X$, where $A$ is an abelian variety, and $X$ is Kobayashi hyperbolic, \cite{D97}.  Requiring that the cotangent bundle be ample is certainly a very strong property, and for a long time there were very few examples of such varieties, although they were expected to be reasonably abundant.  One such example was constructed by Michael Schneider.
\begin{thm}\textup{\cite{s86}}\label{A} 
Let $f : X \to Y$ be a smooth projective nonisotrivial morphism, where $X$ and $Y$ are smooth projective varieties over $\C$, of dimensions 2 and 1, respectively.  Suppose for all $y \in Y$, the Kodaira-Spencer map $\rho_{f,y}: T_{Y,y} \to \coh 1.X_y.T_{X_y}.$ is non-zero, then $\OM_X$ is ample.
\end{thm}

Note that certain Kodaira surfaces satisfy the condition above.  In this paper, we generalize Theorem \ref{A} to varieties of higher dimensions.  To do so,  we will introduce a slightly weaker notion of ampleness, which we will call ``quasi-ample" and ``quasi-ample with respect to an open subset $U$" (see Definitions \ref{defqample} and \ref{qamplewrtU}).   Using this notion we extend Schneider's result to varieties of higher dimension.

\begin{thm} \label{B}
Let
\[ \xymatrix{X^n \ar[r]^{f_n} & X^{n-1} \ar[r]^{f_{n-1}} & X^{n-2} \ar[r]^{f_{n-2}} & \cdots \cdots \ar[r]^{f_3} & X^2 \ar[r]^{f_2} & X^1} \]
where each $X^i$ is a smooth projective variety over $\C$ of dimension $i$, and each $f_i:X^i \to X^{i-1}$ is a smooth, projective morphism with $\var(f_i) = i-1$. Then,  $\OM_{X^n}$ is nef and  quasi-ample with respect to an open $U_n$ (precisely defined below in Theorem \ref{tower1}) and $\sO_{\P(\OM_{X^n})}(1)$ is a big line bundle on $\P(\OM_{X^n})$. 
\end{thm} 

We also extend this result to towers of varieties
\[ \xymatrix{X^n \ar[r]^{f_n} & X^{n-1} \ar[r]^{f_{n-1}} & X^{n-2} \ar[r]^{f_{n-2}} & \cdots \cdots \ar[r]^{f_3} & X^2 \ar[r]^{f_2} & X^1} \]
where the $f_i$  are not necessarily smooth (Theorem \ref{notsmoothtower}), and show that $\OM^1_{X^n}(\log D)$ is quasi-ample with respect to an open set, where $D$ is a suitable divisor taking in account the singularities of the given morphisms.  

In Section \ref{subsec:ex} we construct a tower of varieties satisfying the conditions of Theorem \ref{tower2}  using a construction due to Kodaira which for any $n$ produces a $g$ for which $\sM_g$ contains a complete $n$-dimensional subvariety.  

\noindent \textbf{Acknowledgments.}  Many of the results that appear in this paper originally appeared in my doctoral dissertation at the University of Washington.  I would like to thank my advisor S\'andor Kov\'acs for all his guidance and support.

\subsection{Definitions and Examples}
Recall that a vector bundle $\sE$ on a  proper scheme $X$ is \emph{ample} if for every coherent sheaf $\sF$, there is an integer $m_0 >0$, such that for every $m \geq m_0$, the sheaf $\sF \ten \sym^m \sE$ is generated as an $\sO_X$-module by its global sections.  Equivalently, $\sE$ is ample if the tautological line bundle $\sO_{\P(\sE)}(1)$ on $\P(\sE)$ is ample.  

A vector bundle $\sE$ on a  proper scheme $X$ is \emph{nef (or semipositive)} if for every complete nonsingular curve $C$ and map $\ga: C \to X$, every quotient bundle $\sQ$ of $\ga^* \sE$ has degree at least zero. Equivalently, $\sE$ is nef if the tautological line bundle $\sO_{\P(\sE)}(1)$ on $\P(\sE)$ is nef.  

Ample, respectively nef, vector bundles, have many nice properties.  For example, quotients of ample (respectively nef) vector bundles are ample (respectively nef) and extensions of ample (respectively nef) vector bundles are ample (respectively nef).  For more properties see \cite{H66} or \cite[Sec 6.1A and 6.1B]{LII}.  

Let $\sE$ be a vector bundle of rank $r$ on an irreducible projective variety $X$ of dimension $n$.  Following the work of Fulton and Lazarsfeld \cite{FL} we can introduce a type of numerical positivity.  More precisely, starting with a weighted homogeneous polynomial $P\in \Q[c_1, \ldots, c_r]$, we get a Chern number \[ \int_X P(c(\sE)) := \int_X P(c_1(\sE), \ldots, c_r(\sE)). \]  

\begin{definition} Let $\LA(n,r)$ be the set of all partitions of $n$ by non-negative integers less than or equal to $r$.   Then for every $\la = (\la_1, \ldots, \la_n) \in \LA(n,r)$ we can form the \emph{Schur polynomial}, $s_\la \in \Q[c_1, \ldots, c_r]$, of weighted degree $n$, which is the determinant of the $n \times n$ matrix:
\[
\left| \begin{array}{cccc} c_{\la_1} & c_{\la_1+1} & \ldots & c_{\la_1+n-1}\\ 
c_{\la_2-1} & c_{\la_2} & \ldots & c_{\la_2+n-2}\\
\vdots & \vdots & \vdots & \vdots \\
c_{\la_n-n} & c_{\la_n-n+1} & \ldots & c_{\la_n}\\ \end{array} \right| \]

\noindent where $c_0=1$ and $c_i=0$ if $i \notin [0,r]$.  
\end{definition}
 In particular, if $P$ is a positive linear combination of Schur polynomials and $\sE$ is ample (respectively nef), then $ \int_X P(c(\sE)) > 0$ (respectively $\geq$).

We next generalize the notion of big for vector bundles.  Recall that a divisor $D$ is \emph{big} if there exists a $c>0$, such that $\dimcoh 0.X.\sO_X(mD). > cm^n$ for  $m \gg 1$.  Generalizing this notion to vector bundles is not consistent in the literature.  As a first definition, we generalize the notion as we did with ample and nef.  This is the definition given for example in \cite{LII} and \cite{D05}.  

\begin{definition}\label{big1} \cite[6.1.23]{LII} . Let $\sE$ be a vector bundle on $X$, then $\sE$ is \emph{L-big} if $\sO_{\P(\sE)}(1)$ is a big line bundle on $\P(\sE)$.  (Here, \emph{L-big} is used instead of \emph{big} to avoid confusion with the second definition given below.) \end{definition}

\begin{Ex} \label{Ex1} As an example, consider the rank $2$ vector bundle $\sE= \sO_{\P^1} \oplus \sO_{\P^1}(1)$ on $\P^1$.  Since a direct sum of line bundles is L-big if and only if some $\N-$linear combination of the direct summands is a big line bundle  \cite[2.3.2(iv)]{LII}, we see that $\sE$ is L-big.
Note in this example $\sE$ is L-big but it has a quotient $\sO_{\P^1}$, which is not big.

It is also useful to see that $\sO_{\P(\sE)}(1)$ is big in a slightly different way (which we will refer to in Example \ref{Ex2}).  First note that $\P(\sE) \iso \P(\sE')$ where $\sE' := \sE \ten \sO_{\P^1}(-1) \iso \sO_{\P^1} \oplus \sO_{\P^1}(-1)$. Fix a section $C_0$ of $X=\P(\sE')$ with $\sO_X(C_0) \iso \sO_{\P(\sE')}(1)$, and let $f$ be a fiber.  Then $\sO_{\P(\sE)}(1)$ corresponds to a section $C_1$, which is linearly equivalent to $C_0+f$.    Given $A=aC_0+bf$ ample, then $a>0$ and $b>a$ \cite[V.2.18]{H77}, so that $bC_1 - A = (b-a)C_0$ is effective, and hence $\sO_{\P(\sE)}(1)$ is big.   \end{Ex}

We next turn to a different generalization of big to vector bundles introduced by Viehweg. To do so, we first need a generalization of nef.

\begin{definition}\label{big2}\cite{V83} Let  $\sE$ be vector bundle on a projective variety $X$ and let $\sH$ be an ample line bundle.
\begin{enumerate}
\item $\sE$ is \emph{weakly positive over an open $U$} if for every $a>0$, there exists $b>0$ such that $\sym^{ab}(\sE) \ten \sH^b$ is globally generated over $U$, that is $\coh 0.X.\sym^{ab}(\sE) \ten \sH^b. \ten \sO_X \to \sym^{ab}(\sE) \ten \sH^b$ is surjective over $U$.
\item $\sE$ is \emph{V-big (or ample with respect to $U$)} if there exists an open dense $U$ in $X$ and  $c>0$ such that $\sym ^{c} \sE \ten \sH^{-1}$ is weakly positive over $U$.  
\end{enumerate}
\end{definition}

Equivalently, a vector bundle $\sE$ is ample with respect to $U$ if and only if the tautological bundle $\sO_{\P(\sE)}(1)$ is ample with respect to $\pi^{-1}(U)$ (where $\pi: \P(\sE) \to X$), \cite[3.4]{V89}.  At first glance it seems that these two generalizations of big to vector bundles may be equivalent, however the second notion of big, V-big, is strictly stronger than L-big.  

\begin{Ex} \label{Ex2} Consider the above Example \ref{Ex1}, where we view $\sO_{\P(\sE)}(1)$ as $C_1 = C_0 + f$.  Then $C_1$ is ample with respect to $V=X - C_0$. But $\pi(V)=\P^1$, so there is no open $U \subset \P^1$ with $\pi^{-1}(U)=V$.  Thus $\sE = \sO_{\P^1} \oplus \sO_{\P^1}(1)$ is L-big but not V-big. \end{Ex}

In general, we have that if $\sE \to \sQ$ is surjective over an open set $U$ and $\sE$ is ample with respect to $U$, then $\sQ$ is also ample with respect to $U$ \cite[3.30]{V89}.  As noted in Example \ref{Ex1}, L-big does not have this property, so this again shows that the V-big is stronger than L-big.

\begin{Ex} \label{2ex} As a second example, let $C$ be a curve of genus greater than one and let $X = C \times C$, with projections $p_1$ and $p_2$.  In this case, the cotangent bundle $\OM_X$ is L-big, that is $\O_{\P(\OM_X)}(1)$ is a big line bundle on $\P(\OM_X)$, but since $\OM_X$ surjects onto $p_1^* \om_C$, we see that $\OM_X$ is not V-big.  Similarly, we can see that the two definitions of big are not equivalent for the tangent bundle.  Consider the smooth quadric surface $Q$, then $T_Q$ is L-big, but $T_Q$ surjects onto $p_1^*\om_{\P^1}^{-1}$, so it is not V-big.  \end{Ex}

In view of these different definitions we will avoid saying that a vector bundle $\sE$  on $X$ is ``big'' and instead say that either $\sO_{\P(\sE)}(1)$ is a big line bundle or $\sE$ is ample with respect to some open set $U$.

We next define a new notion of positivity that is slightly weaker than ample, but stronger than nef.

\begin{definition} \label{defqample}
A vector bundle $\sE$  on $X$ is  \emph{quasi-ample} if for every non-constant morphism $\ga: C \to X$ from a complete nonsingular curve $C$, $\ga^*\sE$ is ample on $C$.  
\end{definition}

In the case where $\sE$ is a line bundle, the terminology \emph{strictly nef} has been used \cite{S95}.  
Many properties of ample vector bundles carry over to quasi-ample vector bundles.  For example,

\begin{thm}\label{qampleprop} \textup{\cite[4.3]{J07}}
Let $\sE$ and $\sE'$ be a vector bundles on $X$
\begin{enumerate}
\item If $\sE$ is quasi-ample then any quotient of $\sE$ is quasi-ample.
\item If $\sym^m \sE$ is quasi-ample for some $m$, then $\sE$ is quasi-ample.
\item If $\sE$ is quasi-ample, then $\sym^m\sE$ and $\sE^m$ are quasi-ample for every $m>0$, and $\LA^m\sE$ is quasi-ample for $m=1, 2, \ldots, r$, where $r$ is the rank of $\sE$. 
\item If $\sE$ and $\sE'$ are quasi-ample, then $\sE \ten \sE'$ is quasi-ample.
\item Let 
\[ \sesshort \sE'.\sE.\sE''. \]
be an exact sequence of vector bundles on $X$.
If $\sE'$ and $\sE''$ are quasi-ample, then $\sE$ is quasi-ample.
\item Let $\sE$ be quasi-ample, and let $Y$ be a subscheme of $X$.  Then $\sE|_Y$ is quasi-ample on $Y$.
\end{enumerate}
\end{thm}
\begin{proof}  We give the proof of $(1)$; the others follow similarly.  Let $\sE$ be a quasi-ample vector bundle on $X$ and $\sQ$ a quotient of $\sE$.  If $\ga: C \to X$ is any non-constant morphism from a complete non-singular curve, then $\ga^* \sE$ is an ample vector bundle with quotient $\ga^* \sQ$, so $\ga^* \sQ$ is also ample.  Hence $\sQ$ is quasi-ample.
\end{proof}

\noindent We also have the following criteria for when a quasi-ample bundle is ample, originally due to Gieseker, see \cite[6.1.7]{LII}, for Gieseker's original statement and proof or \cite[4.7]{J07}.

\begin{thm} \label{qaisa} Let $\sE$ be a vector bundle on $X$, where $X$ is proper over a field $k$.  Then $\sE$ is ample if and only if the two following conditions are satisfied:
\begin{enumerate}
\item There exists an $m_0>0$ such that $\sym ^m \sE$ is generated by global sections for all $m \geq m_0$, and
\item $\sE$ is quasi-ample.
\end{enumerate}
\end{thm}

\begin{cor} \label{qamplebigcan} Let $X$ be a complex projective variety with at most canonical singularities.  If $\om_X$ is quasi-ample and big, then $\om_X$ is ample.
\end{cor}
\begin{proof}  Let $K_X$ be a canonical divisor corresponding to $\om_X$.  Since $K_X$ is quasi-ample, it is nef.  Thus $2K_X - K_X = K_X$ is nef and big, so by the Base Point Free Theorem, \cite[3.3]{KM98}, $|bK_X|$ has no base points for $b \gg 0$.  Thus $\om_X^b$ is generated by global sections, so by (\ref{qaisa}) $\om_X$ is ample.
\end{proof}

In general, quasi-ample does not imply ample.  Mumford constructs an example of a quasi-ample vector bundle that is not ample, see \cite[Example 10.6]{H70}.  Namely, he starts with a curve of genus greater than two and a rank two vector bundle $\sE$ of degree zero such that $\sym^m(\sE) $ is stable for all $m \geq 0$.  Then  $\sO_{\P(\sE)}(1)$ is a quasi-ample line bundle that is not big, hence can not be ample.  Ramanujam extends this  example to produce a quasi-ample and big line bundle which is not ample, see \cite[Example 10.8]{H70}.

However, in certain cases quasi-ample implies ample.  For example, it is not difficult to see that for the tangent bundle of a projective variety, $T_X$ being quasi-ample implies that $T_X$ is ample, and hence $X \iso \P^n$.  In the case of the cotangent bundle, it is unknown if $\OM_X$ being quasi-ample implies that $\OM_X$ is ample.  It also is unknown whether $\OM_X$ quasi-ample with $\sO_{\P(\OM_X)}(1)$ big implies that $\OM_X$ is ample with respect to an open set (note in Example \ref{2ex} where, $X = C \times C$, $\OM_X$ is not quasi-ample).

We can weaken the condition of quasi-ample as follows

\begin{definition} \label{qamplewrtU} If $U \subseteq X$ is an open set, $\sE$ \emph{quasi-ample with respect to $U$} if for every non-constant morphism $\ga: C \to X$ from a complete nonsingular curve $C$, where $\ga(C) \cap U \neq \emptyset$, $\ga^*\sE$ is ample on $C$.
\end{definition}

 We end this section by recalling the notions of isotriviality and maximum variation.  
 
 \begin{definition} A morphism $X \to S$, where $S$ is a complete nonsingular curve, is \emph{isotrivial} if $X_s \iso X_t$, for general $s,t \in S$. 
\end{definition}
 
 Note that if $X \to S$ is smooth projective isotrivial morphism then there exists an \'etale cover $S' \to S$  such that $X \times _S S' \to S'$ is trivial.  
If $f: X \to S$ is nonisotrivial, then for general $t \in S$, the Kodaira-Spencer map at $t$, $\rho_{f,t}: T_{S,t} \to \coh 1. X_t. T_{X_t}., $ is nonzero.  
  
More generally, let $f: X \to Y$ be a surjective morphism between smooth projective varieties, then $\var (f)$ denotes the number of effective parameters of the birational equivalence classes of the fibers.  For the rigorous definition of $\var (f)$, see \cite[2.8]{K87} or \cite[pg. 329]{V83}.  If $\var(f)=0$ then $X_y \iso X_t$ for general $y, t \in Y$.  If $\var (f) = \dim Y$, we say that $f$ has maximum variation.  If $f$ is smooth and if for all $y \in Y$ the set $\{ p \in Y | X_p \iso X_y \}$ is finite, then $\var (f) = \dim Y$. Conversely, if $f$ is smooth and  $\var (f) = \dim Y$, then there exists a moduli space for the fibers of $f$, and hence there exists an open set $U \subseteq Y$ such that for all $y \in U$, the set $\{ p \in U | X_p \iso X_y \}$ is finite.  In the case where a moduli space exists for the fibers of $f$, the variation of $f$ is equal to the rank of the Kodaira-Spencer map at a general point of $Y$, so in particular, if $\var (f) = \dim Y$, then the Kodaira-Spencer map at a general point of $Y$ is injective.

\section{Positivity of Cotangent Bundles}

If $X$ is a complex smooth projective variety, we can consider the case where the cotangent bundle $\OM_X$ is ample.  In this case, $X$ is both algebraically and Kobayashi hyperbolic  \cite{D97}, so in particular, $X$ contains no rational or elliptic curves and any map from an abelian variety to $X$ is constant.  If we consider the weaker case of $\OM_X$ being only quasi-ample, $X$ still has some nice properties.

\begin{lemma}
Let $X$ be a smooth projective variety with quasi-ample cotangent bundle $\OM_X$.  Then
\begin{enumerate}
\item If $Y \subset X$ is a nonsingular subvariety, then $Y$ has a quasi-ample cotangent bundle $\OM_Y$.  
\item If $f: Y \to X$ is any morphism where $Y$ is an abelian variety or $\P^1$, then $f$ is constant.
\end{enumerate}
\end{lemma}
\begin{proof} Let $i:Y \into X$, where $X$ has a quasi-ample cotangent bundle.   If $\sI$ denotes the ideal sheaf of $Y$, then we have the short exact sequence 
\[ \sesshort \sI/\sI^2 . \OM_X|_Y . \OM_Y.. \]
Since $\OM_X$ is quasi-ample, $\OM_X|_Y$ is quasi-ample and hence so is $\OM_Y$.

Suppose $f: Y \to X$ is a non-constant morphism from an abelian variety $Y$ to a smooth projective variety $X$, with $\OM_X$ quasi-ample.  Let $C$ be a complete nonsingular curve and let $\ga: C \to Y$ be a non-constant morphism such that  $f\ga: C \to X$ is also non-constant. Then we get  the following commutative diagram
\centerline{
\xymatrix{
\ga^*f^* \OM_X  \ar[r]^{\be} \ar[d]_{\al} &  \OM_C \ar@{=}[d] \\
\ga^*\OM_Y \ar[r]^{\be'} &  \OM_C \\
}}
with $\be$ and $\be'$ nonzero.  Hence $\al: \ga^*f^* \OM_X \to \ga^*\OM_Y$ must also be nonzero.  Since $\OM_X$ is quasi-ample, $\Hom( \ga^*f^* \OM_X, \sO_C)=0$, and since $Y$ is an abelian variety, $\ga^* \OM_Y \iso \sO_C^{\oplus d},$ where $d$ is the dimension of $Y$.  Thus $\Hom( \ga^*f^* \OM_X, \ga^*\OM_Y)=0$ forcing $\al$ to be zero, which is a contradiction, so $f: Y \to X$ must be constant.  If $f: \P^1 \to X$ is non-constant, then we get a non-constant map $\si: f^*\OM_X \to \OM_{\P^1}$.  Let $\sE$ be the image of $\si$, then $\sE$ is ample since it is a quotient of $f^*\OM_X$.  But then, since $T_{\P^1}$ surjects onto $\sE^{\vee}$, $\sE^{\vee}$ is also ample, which leads to a contradiction.  Thus any $f: \P^1 \to X$ must be constant.
\end{proof}

We now begin our generalization of Theorem \ref{A}.

\subsection{Smooth Towers of Smooth Projective Varieties} \label{smoothsec}

\begin{thm}\label{genus}
Let $f:X \to Y$ be a smooth projective morphism, where $\dim X = \dim Y + 1$ and $\var(f) = \dim Y$.  Then, for all $y \in Y$, $X_y$ is a curve of genus at least 2, and if $\dim Y =1$, the genus of $Y$ is also at least 2.  In particular $\om_{X_y}$ is an ample line bundle on $X_y$ for all $y \in Y$.
\end{thm}
\begin{proof}
Suppose first $\dim Y =1$, so $\dim X =2$.  Then, by \cite[III.15.4]{B84}, $g(Y) \geq 2$.  Since $f$ is flat, the genus of the fibers is constant.  If the genus is zero, then all the fibers are isomorphic to $\P^1$, hence $f$ is isotrivial.  By the existence of the $J$-fibration, the genus of the fibers can not be one \cite[Ch V, Sections 9 and 14]{B84}.  Thus the genus of the fibers is at least 2.  
If $\dim Y =n$, for $n>1$, let $S$ be a general curve in $Y$.  Restricting $f$ to $f^{-1}(S)$, we are in the previous case, so $X_y$ must have genus at least 2.\end{proof}

Given a surjective map of smooth projective varieties, $f: X \to Y$, of relative dimension $k$, we will use the positivity of $f_*\om_{X/Y}^m$ and, more generally, $f_*\OM_{X/Y}^k(\log \DE)^m$, where $\DE$ is a normal crossing divisor on $X$.  These deep results are found in the work of Viehweg and Koll\'ar (cf.  \cite{V83II}, \cite{K87}).  In particular, we will use the following formulation

\begin{thm} \label{pushforward}
Let $f: X \to Y$ be a surjective map of smooth projective varieties of relative dimension $k$, with $\var(f) = \dim Y$.
\begin{enumerate}
\item \textup{\cite[3.4]{VZ} } If $f: X \to Y$ is smooth and $\om_{X/Y}$ is $f$-ample, then $f_*\om_{X/Y}^m$ is ample with respect to an open dense $V \subset Y$ for all $m>1$, where $f_*\om_{X/Y}^m \neq 0$.  Furthermore, we can take $V$ to be the open set where the moduli map $\eta: V \to M_h$ is quasi-finite over it's image.
\item \textup{\cite[3.6]{VZ}} Let $S \subset Y$ be a reduced normal crossing divisor containing the discriminant locus and let $\DE:=f^* S$ be a normal crossing divisor.  Let $V:=Y-S$ and $U:=X-\DE$, so we have a smooth family $U \to V$.  Suppose that $\om_{U/V}$ is $f-$semi-ample, the smooth fibers of $f$ are canonically polarized and that the moduli map $\eta: V \to M_h$ is quasi-finite over it's image.  Then for $m$ sufficiently large and divisible, $f_* \OM_{X/Y}^k(\log \DE)^m$ is ample with respect to $V$.
\end{enumerate}
\end{thm}

\begin{thm}\label{amplerel} 
Let $f:X \to Y$ be a surjective nonisotrivial morphism of smooth projective varieties, with $\dim Y =1$ and $\om_{X_y}$ an ample line bundle for all $y\in Y$. Then $\om_{X/Y}$ is an ample line bundle.
\end{thm}
\begin{proof}  The proof follows exactly as in \cite[2.5]{K97}, once we know that $f_*\om_{X/Y}^m$ is ample for some $m>0$, \cite{V83II}, \cite{K87}. \end{proof}

The following is due to Gieseker

\begin{lemma}\textup{\cite[Proposition 2.2]{G71} }\label{gc}
Let $C$ be a nonsingular curve and suppose $\sF$ is ample on $C$ and we have a non-trivial extension \[ \sesshort \sO_C.\sE.\sF.. \] 
Then $\sE$ is ample.
\end{lemma}

We first consider Schneider's original set-up, that is a nonisotrivial smooth projective morphism from a surface to a curve.

\begin{lemma}\label{fiber} Let $X$ and $Y$ be smooth projective varieties, with $\dim Y =1$, and let $f:X \to Y$ be a smooth projective nonisotrivial morphism.   Let $B := \{ p \in Y | \al : \coh 0. X_p. f^*T_Y|_{X_p}. \to \coh 1.X_p. T_{X_p}. $ is not injective$\}$. 
Then for any $y \in Y - B$, the short exact sequence,
\begin{equation}\label{1}
\sesshort f^*\OM_Y|_{X_y}.\OM_X|_{X_y}.\OM_{X/Y}|_{X_y}. \end{equation}
does not split.
\end{lemma}
\begin{proof}
Let $y \in Y-B$, then  $\al : \coh 0. X_p. f^*T_Y|_{X_p}. \to \coh 1.X_p. T_{X_p}. $ is injective.  To show that \eqref{1} does not split, it suffices to show that
\begin{equation}\label{2} 
\sesshort T_{X_y}. T_X|_{X_y}. f^*T_Y|_{X_y}.
\end{equation}
does not split.  Taking cohomology, we get
\[ 
 \cdots \to \coh 0.X_y.T_X|_{X_y}. \to \coh 0.X_y.f^*T_Y|_{X_y}. \to \coh 1.X_y.T_{X_y}. \to  \coh 1.X_y.T_X|_{X_y}. \to \cdots  \]
If \eqref{2} splits, then $\coh 0.X_y.T_X|_{X_y}. \to  \coh 0.X_y.f^*T_Y|_{X_y}.$ is surjective, and so the image of $\al: \coh 0.X_y.f^*T_Y|_{X_y}. \to \coh 1.X_y.T_{X_y}.$ is zero.  Thus $\im(\al) = \ker(\al) =0$, so $ \coh 0.X_y.f^*T_Y|_{X_y}. =0$, a contradiction.  Therefore \eqref{1} must not split. 
\end{proof}

\begin{cor} \label{imagefiber}
 Let $X$ and $Y$ be smooth projective varieties over $\C$ of dimensions $2$ and $1$, respectively,  and let $f:X \to Y$ be a smooth projective nonisotrivial morphism.  Suppose $\ga: C \to X$ is a non-constant morphism from a complete nonsingular curve, with $\ga(C)$ contained in a fiber of $f$, say $X_y$.  Moreover, suppose $\al : \coh 0. X_y. f^*T_Y|_{X_y}. \to \coh 1.X_y. T_{X_y}. $ is  injective.  Then $\ga^* \OM_X$ is an ample vector bundle on $C$.
\end{cor}
\begin{proof}
 Suppose $\ga(C) \subseteq X_y$, for some $y \in Y$. Since $\ga:C \to X_y$ is finite and the pull-back of an ample line bundle by a finite map is ample, it suffices to show that $\OM_X|_{X_y}$ is ample.  By (\ref{fiber})
 \[ \sesshort f^*\OM_Y|_{X_y}.\OM_X|_{X_y}.\OM_{X/Y}|_{X_y}. \]
 does not split.  Since $f^* \OM_Y|_{X_y} \iso \sO_{X_y}$ and $\OM_{X/Y}|_{X_y} \iso \om_{X_y}$ is ample by (\ref{genus}), $\OM_X|_{X_y}$ is ample by (\ref{gc}).
 \end{proof}

\begin{thm}\label{above} Let  $f:X \to Y$ be a smooth, projective, nonisotrivial morphism, with $X$ and $Y$ projective varieties over $\C$, of dimensions $2$ and $1$, respectively.  Let $B:= \{ p \in Y | \al : \coh 0. X_p. f^*T_Y|_{X_p}. \to \coh 1.X_p. T_{X_p}. $ is not injective$\}$. Then, $\OM_X$ is quasi-ample with respect to $U := f^{-1}(Y - B)$ and for all $\la \in \LA (2,2)$, the Schur polynomial is positive, that is $\int_X s_\la (\OM_X) > 0$.  In particular, $\sO_{\P(\OM_X)}(1)$ is big.
\end{thm}

\begin{proof}
We will first show that $\OM_X$ is quasi-ample with respect to $U$. Let $\ga:C \to X$ be a non-constant morphism from a complete nonsingular curve $C$, such that $\ga(C) \cap U \neq \emptyset$. If $\ga(C)$ is in a fiber of $f$, then by (\ref{imagefiber}), $\ga^*\OM_X$ is ample.

Now suppose $\ga(C)$ is not in a fiber of $f$, so $f \circ \ga: C \to Y$ is non-constant. We have the following short exact sequence:
\[ \sesshort \ga^*f^*\om_Y.\ga^*\OM_X.\ga^*\om_{X/Y}.. \]
Since $\om_Y$ is ample, $\ga^*f^*\om_Y$ is ample on $C$.  By (\ref{amplerel}), $\om_{X/Y}$ is ample on $X$, hence $\ga^*\om_{X/Y}$ is ample on $C$.  Thus, $\ga^*\OM_X$ is ample.  Therefore, $\OM_X$ is quasi-ample with respect to $U$.  Note also that $\OM_X$ is an extension of two nef line bundles, namely $f^* \om_Y$ and $\om_{X/Y}$, so $\OM_X$ is nef.

We now will show that the Schur polynomials are positive.  Let $\la \in \LA (2,2)$, then $\la = (1,1)$ or $\la = (2,0)$, and by definition, 
\[ s_{(1,1)}(\OM_X) = c_1(\OM_X)^2 - c_2(\OM_X) \: \text{and} \: s_{(2,0)}(\OM_X) =  c_2(\OM_X). \]
Using the short exact sequence
\[ \sesshort f^*\om_Y.\OM_X.\om_{X/Y}., \]
we see that,
\[ c_1(\OM_X) = c_1(f^*\om_Y) + c_1(\om_{X/Y}) \: \: \text{and} \: \: c_2(\OM_X) = c_1(f^*\om_Y) \cdot c_1(\om_{X/Y}). \]
Thus,
\[ s_{(1,1)}(\OM_X) = c_1(\OM_X)^2 - c_2(\OM_X) =  c_1(f^*\om_Y) \cdot c_1(\om_{X/Y}) +  c_1(\om_{X/Y})^2 \]
and
\[ s_{(2,0)}(\OM_X) =  c_1(f^*\om_Y) \cdot c_1(\om_{X/Y}). \]
 Since $\om_Y$ is ample, $c_1(\om_Y)= \sum m_i[y_i]$, where $y_i \in X$ are points and $\sum m_i>0$. Since $f$ is flat,  $c_1(f^* \om_Y) = f^* c_1(\om_Y) = \sum n_i [X_{y_i}]$.  Now, $\om_{X/Y}|_{X_y}= \om_{X_y}$ is ample for all $y \in Y$ and $\deg (\om_{X_y}) = 2g(X_y) -2 \geq 2$ is constant for all $y \in Y$.  Thus,
\[ \int_X c_1(f^*\om_Y)\cdot c_1(\om_{X/Y}) = \sum n_i(2g-2)>0. \]
Furthermore, $\om_{X/Y}$ is an ample line bundle on $X$, by (\ref{amplerel}), so $\int_X c_1(\om_{X/Y})^2 > 0$.  Thus,
\[ \int_X s_{(1,1)}(\OM_X) = \int_X  c_1(f^*\om_Y) \cdot c_1(\om_{X/Y}) +  c_1(\om_{X/Y})^2 > 0, \]
\[ \int_X s_{(2,0)}(\OM_X) = \int_X c_1(f^*\om_Y) \cdot c_1(\om_{X/Y}) >0. \]
Since $\OM_X$ is nef, to show that $\sO_{\P(\OM_X)}(1)$ is big, it suffices to show that $\sO_{\P(\OM_X)}(1)$ has positive top intersection. But by \cite[Lemma 1.8]{G71}, this is equivalent to showing $\int_X s_{(1,1)}(\OM_X) >0$, hence $\sO_{\P(\OM_X)}(1)$ is big.
\end{proof}

We will now consider the following tower
\[ \xymatrix{X^n \ar[r]^{f_n} & X^{n-1} \ar[r]^{f_{n-1}} & X^{n-2} \ar[r]^{f_{n-2}} & \cdots \cdots \ar[r]^{f_3} & X^2 \ar[r]^{f_2} & X^1} \]
where each $X^i$ is a smooth projective variety over $\C$ of dimension $i$, and for $2 \leq i \leq n $, $f_i:X^i \to X^{i-1}$ is a smooth, projective morphism with $\var(f_i) = \dim X^{i-1}$.

We first prove one generalization

\begin{thm}\label{schurtower}
Let
\[ \xymatrix{X^n \ar[r]^{f_n} & X^{n-1} \ar[r]^{f_{n-1}} & X^{n-2} \ar[r]^{f_{n-2}} & \cdots \cdots \ar[r]^{f_3} & X^2 \ar[r]^{f_2} & X^1} \]
where each $X^i$ is a smooth projective variety over $\C$ of dimension $i$, and each $f_i:X^i \to X^{i-1}$ is a smooth, projective morphism with $\var(f_i)=\dim X^{i-1}$. Then,  $\OM_{X^n}$ is nef and for all $\la \in \LA (n,n)$, the corresponding Schur polynomial is positive, that is $\int_X s_\la (\OM_{X^n}) > 0$.  In particular $\sO_{\P(\OM_{X^n})}(1)$ is big.
\end{thm}

\begin{proof}
We will prove this by induction on $n$.  By (\ref{above}), the statement is true for $n=2$, and we assume it holds for $n-1$.  Let $X:=X^n$, then we have the short exact sequence
\[ 0 \to f_n^* \OM_{X^{n-1}} \to \OM_X \to \om_{X/X^{n-1}} \to 0. \]
Now $ \OM_{X^{n-1}}$ is nef by induction, so $f_n^* \OM_{X^{n-1}}$ is nef.  Since $f_n$ is smooth and $\om_{X/X^{n-1}}$ is $f_n$-ample, $(f_n)_* \om_{X/X^{n-1}}^m$ is nef for all $m>0$, by \cite[6.22]{V95}.  Additionally, since $\om_{X/X^{n-1}}$ is $f_n$-ample,  for $m \gg 0$, the natural map 
\[ f_n^*(f_n)_*\om_{X/X^{n-1}}^m \to \om_{X/X^{n-1}}^m \]
is surjective.  Thus $\om_{X/X^{n-1}}$ is nef.  Hence $\OM_X$ is nef.

Let $\la \in \LA(n,n)$.  Define $d_i:=c_i(\OM_X)$, $\al_i := c_i ( f_n^*\OM_{X^{n-1}})$ and $\be := c_1(\om_{X/X^{n-1}})$.  Then from the short exact sequence 
\[ \sesshort f_n^*\OM_{X^{n-1}}.\OM_X.\om_{X/X^{n-1}}., \]
we have, 
\[ d_i = \al_i +\al_{i-1}\be,\: \text{where $\al_0=1$ and $\al_i=0$ for $i \notin [0,n-1].$} \]
Thus, by \cite[Exercise 4 in 5.2]{F97} ,
\begin{eqnarray*}
s_\la(\OM_X) & = & s_\la (d_1, \ldots, d_n) \\
& = & s_\la (\al_1, \ldots, \al_{n-1}, \be) \\
& = & \sum_{\mu \subset \la} s_{\la / \mu}(\al_1, \ldots, \al_{n-1})s_\mu(\be). \\
\end{eqnarray*} 
Now, $\be = c_1(\om_{X/X^{n-1}})$, so if $\mu = (\mu_1, \ldots, \mu_n)$ and $\mu_1 \neq 1$, then $s_\mu(\be)=0$.  Also, note that $s_{(1^k)}(\om_{X/X^{n-1}}) = c_1(\om_{X/X^{n-1}})^k$.  Hence,
\[ s_\la(\OM_X)  = \sum_{k=1}^{n} s_{\la / (1^k)}(f_n^*\OM_{X^{n-1}}) \cdot c_1(\om_{X/X^{n-1}})^k. \]  
Let $1 \leq k \leq n-1$ and $\mu \in \LA(n-k, n-k)$, then 
\[ s_\mu \cdot s_{(1^k)} = \sum_\nu c_{\mu / (1^k)}^\nu s_{\nu} \: \: \text{and} \: \: s_{\nu /(1^k)} = \sum_\mu  c_{\mu / (1^k)}^\nu s_\mu.\]
By \cite[Proposition in 1.1]{F97}, $c_{\mu / (1^k)}^\nu = 1$ if $\nu$ can be obtained from $\mu$ by adding $k$ boxes, no two of which are in the same row, and zero otherwise.  Thus, setting $\nu = \la$,
\[s_{\la/(1^k)} = \sum_\mu s_\mu, \]
where the sum is taken over $\mu \in \LA(n-k, n-k)$ such that $\mu$ can be obtained from $\la$ by subtracting $k$ boxes, no two of which are in the same row.  Therefore,
\[ s_\la(\OM_X)  = \sum_{k=1}^{n} \left( c_1(\om_{X/X^{n-1}})^k \sum_{\mu \in \LA(n-k,n-k)} s_\mu(f_n^*\OM_{X^{n-1}}) \right), \]
where the second sum is taken over $\mu \in \LA(n-k, n-k)$ such that $\mu$ can be obtained from $\la$ by subtracting $k$ boxes, now two of which are in the same row.

Consider the first term, $c_1(\om_{X/X^{n-1}}) \sum_{\mu} s_\mu(f_n^*\OM_{X^{n-1}}).$  By induction, for $\mu \in \LA(n-1,n-1)$, $s_\mu (\OM_{X^{n-1}})$ is a positive polynomial, that is  $ s_\mu (\OM_{X^{n-1}}) = \sum m_iy_i$, where $y_i \in X^{n-1}$ are points and $\sum m_i >0$.  Now, $\om_{X/X^{n-1}}|_{X_y} \iso \om_{X_y}$ is ample on $X_y$ for all $y \in X^{n-1}$, thus
\[ \int_X c_1(\om_{X/X^{n-1}}) \cdot s_\mu (f_n^*\OM_{X^{n-1}}) >0. \]

Since $f_n^*\OM_{X^{n-1}}$ and  $\om_{X/X^{n-1}}$ are nef, for $2 \leq k \leq n-1$ and $\mu \in \LA(n-k, n-k)$ as above,
\[\int_X c_1(\om_{X/X^{n-1}})^k \cdot s_{\mu}(f_n^*\OM_{X^{n-1}}) \geq 0,\]
and
\[\int_X c_1(\om_{X/X^{n-1}})^n \geq 0. \]
Therefore, $\int_X s_\la(\OM_X) > 0$.   In particular, this holds for $\la = (1^n)$, so by  \cite[Lemma 1.8]{G71} $\sO_{\P(\OM_{X^n})}(1)$ is big.
\end{proof}

We continue to assume we have a tower of varieties
\[ \xymatrix{X^n \ar[r]^{f_n} & X^{n-1} \ar[r]^{f_{n-1}} & X^{n-2} \ar[r]^{f_{n-2}} & \cdots \cdots \ar[r]^{f_3} & X^2 \ar[r]^{f_2} & X^1} \]
where each $X^i$ is a smooth projective variety over $\C$ of dimension $i$, and each $f_i:X^i \to X^{i-1}$ is a smooth, projective morphism with the property that $\var (f_i) = \dim X^{i-1}$.  For $1 \leq i \leq n-1$, define
\[ B_i := \{ p \in X^i|\: \al_i: \coh 0. X^{i+1}_p.f_{i+1}^*T_{X^i}|_{X_p^{i+1}}. \to \coh 1.X^{i+1}_p. T_{X^{i+1}_p}. \text{ is not injective} \} \subset X^i \]
Note, that since each $f_{i+1}$ is of maximum variation, $X^i - B_i$ is an open dense set.  Set $U_1 := X^1$, and for  $2 \leq i \leq n$ define open sets $U_i \subset X_i$ as 
\[U_i := f_i^{-1}(U_{i-1}  - B_{i-1}). \]

\begin{lemma} \label{moduli} In the above setting, the moduli map $\mu_{i-1}: V_{i-1}:= U_{i-1} - B_{i-1} \to M_g$ induced by the family $f_i: X^i \to X^{i-1}$ is quasi-finite onto its image.
\end{lemma}
\begin{proof}  It suffices to show that $F_{i-1}: = \{ y \in X^{i-1} | X^i_y \iso X^i_p \text{ for infinitely many } p \in X^{i-1} \}\subseteq B_{i-1}$.  Let $y \in F_{i-1}$, then there exists a connected closed subscheme $Z \subset X^{i-1}$ with $y \in Z$ such that $f_Z: X_Z:= X^i \times_{X^{i-1}} Z \to Z$ has variation $0$.  In particular, $\coh 0. (X_Z)_y. f_Z^*T_Z|_{(X_Z)_y}. \to \coh 1.  (X_Z)_y. T_{(X_Z)_y}.$ is not injective.  Then from the following commutative diagram

\centerline{
\xymatrix{ \coh 0. (X_Z)_y. f_Z^*T_Z|_{(X_Z)_y}. \ar[r] \ar@{^(->}[d] &\coh 1.  (X_Z)_y. T_{(X_Z)_y}. \ar@{=}[d] \\
\coh 0. X_y. f_i^*T_Y|_{X_y}. \ar[r] & \coh 1.  X_y. T_{X_y}.
}}
\noindent we see that $\coh 0. X_y. f_i^*T_Y|_{X_y}. \to \coh 1.  X_y. T_{X_y}.$ can not be injective, hence $y \in B_{i-1}$.
\end{proof}

\begin{thm}\label{tower1} In the above setting, $\OM_{X^n}$ is quasi-ample with respect to $U_n$.
\end{thm}
\begin{proof} We will prove this by induction on $n$.  By (\ref{above}), the statement is true for $n=2$, and we assume it holds for $n-1$.  Let $X:=X^n$, then we have the short exact sequence
\[ 0 \to f_n^* \OM_{X^{n-1}} \to \OM_X \to \om_{X/X^{n-1}} \to 0. \]
Let $\ga: C \to X$ be a non-constant morphism from a complete nonsingular curve $C$ such that $\ga(C)\cap U_n \neq \emptyset$.  Suppose first that $\ga(C)$ is contained in a fiber of $f_n$, say $\ga(C) \subseteq X_y$ for some $y \in X^{n-1}$.  Since $\ga: C \to X_y$ is finite, it suffices to show that $\OM_X|_{X_y}$ is ample.  Note also that since  $\ga(C)\cap U_n \neq \emptyset$, $y \in U_{n-1}  - B_{n-1}$.  Let $f_2 f_3 \cdots f_{n-1}(y)=s\in X^1$, and let $h=f_2 f_3 \cdots f_{n-1}f_n: X=X^n \to X^1$. Then we have the following short exact sequence
\begin{equation} \label{5}
\sesshort T_{X_s}|_{X_y}.T_{X}|_{X_y}.h^*T_{X^1}|_{X_y}.. 
\end{equation}

I claim that (\ref{5}) doesn't split.  Indeed, suppose it splits, then looking at the long exact sequence
\[ \ldots \to \coh0.X_y.T_{X}|_{X_y}. \to \coh0.X_y.h^*T_{X^1}|_{X_y}. \to \coh1.X_y.T_{X_s}|_{X_y}. \to \coh1.X_y.{T_X}|_{X_y}. \to \ldots \]
we have a surjection $\be: \coh0.X_y.T_{X}|_{X_y}. \onto \coh 0.X_y.h^*T_{X^1}|_{X_y}..$ \
Consider the following commutative diagram,

\centerline{ \xymatrix{ & & & 0 \ar[d] & \\
& 0 \ar[d] & & f_n^*T_{X^{n-1}/X^1}|_{X_y} \ar[d] & & \\
0 \ar[r] & \ar[r] T_{X/X^{n-1}}|_{X_y} \ar[d]\ar[r] &T_X|_{X_y} \ar@{=}[d] \ar[r] & f_n^*T_{X^{n-1}}|_{X_y} \ar[d] \ar[r] & 0 \\
0 \ar[r] & \ar[r] T_{X/X^1}|_{X_y} \ar[d] \ar[r] &T_X|_{X_y} \ar[r] & h^*T_{X^{1}}|_{X_y} \ar[d] \ar[r] & 0 \\
& f_n^*T_{X^{n-1}/X^1}|_{X_y} \ar[d] & & 0 & \\
& 0 &&& \\
}}

\noindent Then taking cohomology, gives

\centerline{\xymatrix{ \ldots \ar[r] & \coh 0.X_y.T_X|_{X_y}. \ar@{=}[d] \ar[r] & \coh 0.X_y. f_n^*T_{X^{n-1}}|_{X_y}. \ar[d] \ar[r]^{\al_{n-1}} & \coh 1. X_y. T_{X/X^{n-1}}|_{X_y}. \ar[d] \ar[r] & \ldots \\
\ldots \ar[r]  & \coh 0.X_y.T_X|_{X_y}. \ar[r]^{\be} & \coh 0.X_y. h^*T_{X^1}|_{X_y}. \ar[r] & \coh 1. X_y. T_{X/X^{1}}|_{X_y}.  \ar[r] & \ldots \\ }}

\noindent Since $y \notin B_{n-1}$, $\al_{n-1}$ is injective, and hence $\im \be = 0$.  But $\be$ is surjective, so  $\coh 0.X_y. h^*T_{X^1}|_{X_y}. = 0,$ a contradiction.  Thus (\ref{5}) does not split, and hence neither does 
\[ \sesshort \sO_{X_y}.\OM_X|_{X_y}.\OM_{X_s}|_{X_y}.. \]
By (\ref{gc}), to show that $\OM_X|_{X_y}$ is ample, it suffices to show that $\OM_{X_s}|_{X_y}$ is ample.

We have
\[ \xymatrix@1{X_s = X^n_s \ar[r]^(.55){(f_n)_s} & X^{n-1}_s \ar[r]^{(f_{n-1})_s}  & \cdots  \cdots \ar[r] & X^3_s  \ar[r]^{(f_3)_s} & X^2_s}, \]
where each $X^i_s$ is a smooth projective variety over $\C$ of dimension $i-1$, and each $(f_i)_s:X^i_s \to X^{i-1}_s$ is a smooth projective morphism with the property that $\var(f_i)_s = \dim(X^{i-1})_s$.   
For $2 \leq i \leq n-1$, define
\[ B_i,s := \{ p \in X^{i}_s |\: \al_{i,s}: \coh 0. X^{i+1}_p.(f_{i+1})_s^*T_{X^i_s}|_{X_p^{i+1}}. \to \coh 1.(X^{i+1}_s)_p.T_{(X^{i+1}_s)_p}.  \text{ is not injective}  \}. \]
 Define open $U_{i,s} \subseteq (X^i)_s$ as follows:
\[ U_{2,s} :=  X^2_s \]
\[ U_{i,s} := (f_i)_s^{-1}(U_{i-1,s} -B_{i-1,s}), \text{ for $3 \leq i \leq n$ }.\]
Then, by induction, $\OM_{X_s}$ is quasi-ample with respect to $U_{n,s} := (f_n)^{-1}_s(U_{n-1,s}  - B_{n-1,s})$. 
 I first claim that $B_{i,s} \subseteq B_i \cap X^i_s.$  Indeed, we have the commutative diagram
 
 \centerline{\xymatrix{ 0 \ar[r] & T_{X^{i+1}_s/X^i_s}|_{X^{i+1}_p} \ar@{=}[d] \ar[r] & T_{X^{i+1}_s}|_{X^{i+1}_p} \ar[r] \ar@{^{(}->}[d] & (f_{i+1})_s^* T_{X^i_s}|_{X^{i+1}_p} \ar[r] \ar@{^{(}->}[d]  & 0 \\
  0 \ar[r] & T_{X^{i+1}/X^i}|_{X^{i+1}_p}  \ar[r] & T_{X^{i+1}}|_{X^{i+1}_p} \ar[r] & f_{i+1}^* T_{X^i}|_{X^{i+1}_p} \ar[r]   & 0 \\ }}

\noindent So taking cohomology, gives

 \centerline {\xymatrix{ \coh 0.X^{i+1}_p. (f_{i+1})_s^* T_{X^i_s}|_{X^{i+1}_p}. \ar@{^{(}->}[d]  \ar[r]^(.6){\al_{i,s}} & \coh  1.X^{i+1}_p. T_{X^{i+1}_p}. \ar@{=}[d] \\
 \coh 0.X^{i+1}_p. f_{i+1}^* T_{X^i}|_{X^{i+1}_p}.   \ar[r]^(.6){\al_i} & \coh  1.X^{i+1}_p. T_{X^{i+1}_p}.
 }}

\noindent Thus if $\al_i$ is injective, $\al_{i,s}$ is injective, so $B_{i,s} \subseteq B_i \cap X^i_s.$

I next claim $(X^i)_s \cap U_i \subseteq U_{i,s} $ for $i \geq 2$.  Indeed, if $i=2$, this follows from the definition of $U_2$.  Suppose it is true for $i-1$, then 
\[ U_i \cap X^i_s = (f_i)_s^{-1}((U_{i-1} \cap X^{i-1}_s) -  (B_{i-1} \cap X^{i-1}_s)) \subseteq (f_i)_s^{-1}(U_{i-1,s} - B_{i-1,s}) = U_{i,s}. \]
Thus, since $y \in X_s \cap (U_{n-1} - B_{n-1}) \subseteq U_{n-1,s} - B_{n-1,s} $, we conclude that $\OM_{X_s}|_{X_y}$ is ample.

Now suppose $\ga(C)$ is not in the fiber of $f_n$, so $f_n \circ \ga: C \to X^{n-1}$ is non-constant. Since $\ga(C) \cap f_n^{-1}(U_{n-1} -  B_{n-1}) \neq \emptyset,$ we have $f_n \ga(C) \cap U_{n-1} \neq \emptyset $.  Thus by induction $\ga^*(f_n^*\OM_{X^{n-1}})$ is ample.  We have the following short exact sequence
\[ \sesshort \ga^*(f_n^*\OM_{X^{n-1}}).\ga^* \OM_X. \ga^*\om_{X/X^{n-1}}.  \]
so to show $\ga^*\OM_X$ is ample, it suffices to show that $\ga^*\om_{X/X^{n-1}}$ is ample.  By (\ref{pushforward}), $f_*\om_{X/X^{n-1}}^m$ is ample with respect to $X_{n-1} - B_{n-1}$, for $m>1$ where 
 $f_*\om_{X/X^{n-1}}^m \neq 0$.  Thus, since $\ga(C) \cap f_n^{-1}(U_{n-1} -  B_{n-1}) \neq \emptyset,$
we have that $\ga^*f^*f_*\om_{X/X^{n-1}}^m$ is ample with respect to $\ga^{-1}f^{-1}(X_{n-1}-B_{n-1})$, an open dense subset of the curve $C$.  Hence, $\ga^*f^*f_*\om_{X/X^{n-1}}^m$ is ample on $C$.  Furthermore, since  $\om_{X/X^{n-1}}$ is $f_n$-ample,
\[ \ga^*f_n^*(f_n)_*\om_{X/X^{n-1}}^m \to \ga^*\om_{X/X^{n-1}}^m \]
is surjective for sufficiently large $m$, and so $\ga^*\om_{X/X^{n-1}}$ is ample.
\end{proof}

\begin{cor} \label{tower2}
Let
\[ \xymatrix{X^n \ar[r]^{f_n} & X^{n-1} \ar[r]^{f_{n-1}} & X^{n-2} \ar[r]^{f_{n-2}} & \cdots \cdots \ar[r]^{f_3} & X^2 \ar[r]^{f_2} & X^1} \]
where each $X^i$ is a smooth projective variety over $\C$ of dimension $i$, and each $f_i:X^i \to X^{i-1}$ is a smooth, projective morphism with the property that for all $y \in X^{i-1}$, $\al_i: \coh 0. X^{i+1}_y.f_{i+1}^*T_{X^i}|_{X_y^{i+1}}. \to \coh 1.X^{i+1}_y. T_{X^{i+1}_y}.$ is  injective.
Then,  $\OM_{X^n}$ is quasi-ample and for all $\la \in \LA (n,n)$, the Schur polynomial is positive, that is $\int_X s_\la (\OM_X) > 0$.  In particular $\OM_{X^n}$ is quasi-ample and  $\sO_{\P(\OM_{X^n})}(1)$ big.
\end{cor}

\begin{proof}  
By the assumption on each $f_i:X^i \to X^{i-1}$, the sets $B_i \subset X^i$ are empty.  Hence each $U_i = X^i$, so by the above theorem (\ref{tower1}), $\OM_{X^n}$ is quasi-ample.  The second statement follows from (\ref{schurtower}).
\end{proof}

Let us also note that the condition on the Kodaira-Spencer maps is necessary for the cotangent bundle is quasi-ample on \emph{all} of $X^n$.  As an example, consider a nonisotrivial smooth projective morphism $f: X \to Y$ from a smooth projective surface to a smooth projective curve.  Suppose there exists $y \in Y$ such that the Kodaira-Spencer map $\rho_{f,y}: T_{Y,y} \to \coh1.X_y. T_{X_y}.$ is zero.  Then $\OM_X|_{X_y} \iso \sO_{X_y} \oplus \OM_{X_y}$, so $\OM_{X}|_{X_y}$ is not ample.

\subsection{Towers of Varieties where the Morphisms are not Smooth} \label{notsmoothsec}

We next weaken the hypothesis on the $f_i$.   Let $X$ be a smooth variety of dimension $n$ and $D \subset X$ a reduced normal crossing divisor.  Recall, that  $\OM_X^1(\log D)$ is the \emph{sheaf of one-forms on $X$ with logarithmic poles along $D$} and is defined as follows: if $z_1, \ldots, z_n$ are local analytic coordinates on $X$, with $D=(z_1\cdots z_l)$, then $\OM_X^1(\log D)$ is locally generated by $\frac {dz_1} {z_1}, \ldots \frac {dz_l} {z_l}, dz_{l+1}, \ldots, dz_n$.  If $D$ has normal crossings, but is not reduced, we abuse notation and write $\OM_X^1(\log D)$ for $\OM_X^1(\log D_{\red})$.  

\begin{lemma} \label{2div} Let $D = D_1 + D_2$ be a normal crossing divisor on a smooth variety $X$.  Suppose $\OM_X^1(\log D_2)$ is quasi-ample with respect to $X - D_2$, then $\OM_X^1(\log D)$ is quasi-ample with respect to $X -D.$
\end{lemma}
\begin{proof}  Let $\ga: C \to X$ be a non-constant morphism from a complete nonsingular curve such that $\ga(C) \cap (X - D) \neq \emptyset$.  Without loss of generality, we may assume $D_1$ and $D_2$ do not contain any common components.  If $\ga(C) \cap D_1 = \emptyset$, then $\ga^* \OM_X^1 (\log D_2) \iso \ga^* \OM_X^1 \log (D)$ is ample.  If $\ga(C) \cap D_1 \neq \emptyset$, then since $\ga(C) \nsubseteq D_1$, $\ga(C) \cap D_1$ must consist of a finite number of points.  Hence we have the short exact sequence 
\[\sesshort \ga^*\OM_X^1(\log D_2). \ga^*\OM_X^1(\log D). \ga^*\sO_{D_1}(D_2|_{D_1}). \]
with $ \ga^*\OM_X^1(\log D_2)$ and $ \ga^*\sO_{D_1}(D_2|_{D_1})$ ample.  Thus $\ga^*\OM_X^1(\log D)$ is also ample.
\end{proof}

We consider the following:
\[ \xymatrix{X^n \ar[r]^{f_n} & X^{n-1} \ar[r]^{f_{n-1}} & X^{n-2} \ar[r]^{f_{n-2}} & \cdots \cdots \ar[r]^{f_3} & X^2 \ar[r]^{f_2} & X^1} \]
where each $X^i$ is a smooth projective variety over $\C$ of dimension $i$. Suppose for $2 \leq i \leq n $,
\begin{enumerate}
\item  $f_i:X^i \to X^{i-1}$ is a flat projective morphism with $\var(f_i) = \dim X^{i-1}$,

\item $\om_{X^i_y}$ is ample for all $y \notin \DE_{i-1,i}$, where we define\[ \DE_{i-1,i} := \{ y \in X^{i-1} | f_i^{-1}(y) \text{ is singular} \}. \]
\end{enumerate}
Put $\DE_{2,1} := f_2^*\DE_{1,2}$, and define recursively for all $i <n$,
\[ \DE_{i, i-1} := f_{i}^*((\DE_{i-1,i-2} + \DE_{i-1,i})_{\red}) \subset X^{i}.\]
We will also write $\DE_{i-1}:= (\DE_{i-1,i-2} + \DE_{i-1,i})_{\red} \subset X^{i-1}$, so that $\DE_{i, i-1} = f_{i}^*(\DE_{i-1})$.
Let $\DE_{1,i}$ be the discriminant locus of $f_2 \circ \ldots \circ f_i: X^i \to X^1$, that is, $\DE_{1,i} := \{s \in X^1 | f_i^{-1} \cdots f_2^{-1}(s) \text{ is singular} \}$, and $\DE_{i,1}:= f_i^* \cdots f_2^* \DE_{1,i}$.  

Define
\[B_i' := \{p \in X^i - \DE_{i, i+1} |\al_i: \coh 0. X^{i+1}_p.f_{i+1}^*T_{X^i}|_{X_p^{i+1}}. \to \coh 1.X^{i+1}_p. T_{X^{i+1}_p}. \text{ is not injective} \}. \] 
Set $B_1:= B_1'$ and for $2 \leq i \leq n$, set

\[ B_i := B_i' + f_i^{-1}(B_{i-1}).\]

\noindent We will show that in the above setting if $\DE_i$ and $\DE_{i+1, i}$ are normal crossing divisors for $1 \leq i \leq n-1$, then $\OM_{X^n}^1(\log \DE_{n,n-1})$ is quasi-ample with respect to $X^n - \DE_{n,n-1} - f_n^{-1}(B_{n-1})$.  We follow the same ideas as before and will prove this by induction.  To prove the case where $n=2$, we need the following lemma.

\begin{lemma} \textup{\cite[1.4]{V00}} \label{notsmoothcurve}
Let $f:X \to Y$ be a nonisotrivial morphism between smooth projective varieties of dimensions 2 and 1, respectively.  Let $\DE_{1,2} \subset Y$ be the discriminant divisor, and $\DE_{2,1}:=f^*\DE_{1,2}$.  Then $2g(Y) -2 + \deg \DE_{1,2} \geq 1$.
\end{lemma}

\begin{thm} \label{notsmoothbase}
Let $f:X \to Y$ be a flat nonisotrivial morphism between smooth projective varieties over $\C$ of dimensions 2 and 1, respectively.  Let $\DE_{1,2} \subset Y$ be the discriminant divisor, and $\DE_{2,1}:=f^*\DE_{1,2}$.  Let $B:=B_1$ be as above.  Suppose $\om_{X_y}$ is ample for all $y \notin \DE_{1,2}$, and both $\DE_{1,2}$ and $\DE_{2,1}$ are normal crossing divisors, then $\OM_X^1(\log \DE_{2,1})$ is quasi-ample with respect to $X-\DE_{2,1}-f^{-1}(B)$.
\end{thm}
\begin{proof}  We have the following short exact sequence 
\[ \sesshort f^* \OM_Y^1(\log \DE_{1,2}). \OM_X^1(\log \DE_{2,1}). \OM_{X/Y}^1(\log \DE_{2,1}). \]Since $ \OM_{X/Y}^1(\log \DE_{2,1})$ is locally free, taking determinants gives $\OM_{X/Y}^1(\log \DE_{2,1}) \iso \om_{X/Y} \ten \sO_Y((\DE_{2,1})_{red} - \DE_{2,1}) \subseteq \om_{X/Y}$.  By (\ref{notsmoothcurve}), $\OM_Y^1(\log \DE_{1,2}) \iso \om_Y \ten \sO_Y(\DE_{1,2})$ is an ample line bundle on $Y$.

Let $\ga: C \to X$ be a non-constant morphism from a complete nonsingular curve $C$, such that $\ga(C) \cap(X - \DE_{2,1}-f^{-1}(B)) \neq \emptyset$.  Suppose first that $\ga(C)$ is contained in a fiber of $f$, say $\ga(C) \subseteq X_y$. Since $\ga: C \to X_y$ is finite, it suffices to show that $\OM_X|_{X_y}$ is ample.  Note that $y \notin \DE_{1,2}$ and hence $X_y \cap \DE_{2,1} = \emptyset$.  Since $y \notin B$, the same argument as in (\ref{imagefiber}) shows that the short exact sequence
\[ \sesshort T_{X/Y}|_{X_y}.T_X|_{X_y}.f^*T_Y|_{X_y}. \]
does not split.  Hence, 
\[ \sesshort f^*\OM_Y^1(\log \DE_{1,2})|_{X_y}.\OM_X^1(\log \DE_{2,1})|_{X_y}.\OM_{X/Y}^1(\log \DE_{2,1})|_{X_y}. \]
does not split. Now, $f^*\OM_Y^1(\log \DE_{1,2})|_{X_y} \iso \sO_{X_y}$, and $\OM_{X/Y}^1(\log \DE_{2,1})|_{X_y} \iso \om_{X_y}$ is ample, since $y \notin \DE_{1,2}$.  Thus, by (\ref{gc}), $ \OM_X|_{X_y}$ is an ample vector bundle on $X_y$.

Next suppose that $\ga(C)$ is not contained in a fiber of $f$, so $f \ga: C \to Y$ is non-constant.  Since $\om_Y \ten \sO_Y(\DE_{1,2})$ is ample, $\ga^* f^*( \om_Y \ten \sO_Y(\DE_{1,2}))$ is also ample.  To show that $\ga^* \OM_X^1(\log \DE_{2,1})$ is ample, it suffices to show that $\ga^* \OM_{X/Y}^1(\log \DE_{2,1})$ is ample.  We have a smooth family $X - \DE_{2,1} - f^{-1}B \to Y - \DE_{1,2} - B$ which satisfies the conditions of (\ref{pushforward}(2)), thus for $m$ sufficiently large and divisible $f_*\OM_{X/Y}^1(\log \DE_{2,1} + f^{-1}B)^m$ is ample with respect to $Y - \DE_{1,2} - B$.  Furthermore, for $m$ sufficiently large
\[ f^*f_*\OM_{X/Y}^1(\log \DE_{2,1} + f^{-1}B)^m \to \OM_{X/Y}^1(\log \DE_{2,1} + f^{-1}B)^m \]
is surjective over $X - \DE_{2,1} - f^{-1}B$.  Since $\ga(C) \cap(X - \DE_{2,1}-f^{-1}(B)) \neq \emptyset$, we find that $\ga^*\OM_{X/Y}^1(\log \DE_{2,1} + f^{-1}B)^m$ is ample with respect to a dense open set of the curve $C$, and hence that $\ga^*\OM_{X/Y}^1(\log \DE_{2,1} + f^{-1}B)$ is an ample line bundle on $C$. Since $f^{-1}B$ is reduced, $\ga^*\OM_{X/Y}^1(\log \DE_{2,1} + f^{-1}B) \iso \ga^*\OM_{X/Y}^1(\log \DE_{2,1})$ is also ample.  

Thus, we have the following short exact sequence
\[ \sesshort \ga^*f^* \OM_Y^1(\log \DE_{1,2}). \ga^*\OM_X^1(\log \DE_{2,1}). \ga^*\OM_{X/Y}^1(\log \DE_{2,1})., \]
with the outer terms ample, so $\ga^*\OM_X^1(\log \DE_{2,1})$ is ample.
Therefore, $\OM_X^1(\log \DE_{2,1})$ is quasi-ample with respect to $X - \DE_{2,1}-f^{-1}(B)$.
\end{proof}

We now prove the general case

\begin{thm}\label{notsmoothtower} Let
\[ \xymatrix{X^n \ar[r]^{f_n} & X^{n-1} \ar[r]^{f_{n-1}} & X^{n-2} \ar[r]^{f_{n-2}} & \cdots \cdots \ar[r]^{f_3} & X^2 \ar[r]^{f_2} & X^1} \]
where each $X^i$ is a smooth projective variety over $\C$ of dimension $i$. Suppose for $2 \leq i \leq n $,
\begin{enumerate}
\item  $f_i:X^i \to X^{i-1}$ is a flat projective morphism with $\var(f_i) = \dim X^{i-1}$,

\item $\om_{X^i_y}$ is ample for all $y \notin \DE_{i-1,i}:= \{ y \in X^{i-1} | f_i^{-1}(y) \text{ is singular} \}.$
\end{enumerate}
Then, if $\DE_i$ and $\DE_{i+1, i}$ (as defined above) are normal crossing divisors, the sheaf $\OM_{X^n}^1(\log \DE_{n,n-1})$ is quasi-ample with respect to $X^n - \DE_{n,n-1} - f_n^{-1}(B_{n-1}).$
\end{thm}
\begin{proof}  We will prove this by induction on $n$.  By (\ref{notsmoothbase}), the theorem is true for $n=2$.  Suppose it is true for $n-1$.  Let $X:=X^n$, then we have the short exact sequence
\[ \sesshort f_n^* \OM_{X^{n-1}}^1(\log \DE_{n-1}). \OM_X^1(\log \DE_{n,n-1}). \OM_{X/X^{n-1}}(\log \DE_{n,n-1}).. \]
Let $\ga: C \to X$ be a non-constant morphism from a complete nonsingular curve $C$ such that $\ga(C)\cap (X - \DE_{n,n-1} - f_n^{-1}(B_{n-1})) \neq \emptyset$.  Suppose first that $\ga(C)$ is contained in a fiber of $f_n$, say $\ga(C) \subseteq X_y$ for some $y \in X^{n-1}$. We must show that $\OM_X(\log \DE_{n,n-1})|_{X_y}$ is ample.  Note, that $y \notin B_{n-1} \cup \DE_{n-1}$, and in particular, $X_y \cap \DE_{n,n-1} = \emptyset$.

Note, $\DE_{i,1} \subseteq \DE_{i,i-1}$ for all $i$, since if $t \in \DE_{1,i}$, then $f_i^{-1} \cdots f_2^{-1}(t)$ is singular, hence there exists $y \in X^j_t$ for some $1 \leq j \leq i-1$, such that $X^{j+1}_y$ is singular, that is $y \in \DE_{j,j+1}.$  Thus, $\DE_{i,1} \subseteq f_i^*\DE_{i-1} = \DE_{i,i-1}$.  
So if we define $s:=f_2f_3 \cdots f_{n-1}(y) \in X^1$, then $s \notin \DE_{i,1}$ for all $i \leq n$, so each $X^i_s$ is a smooth projective variety over $\C$ of dimension $i-1$.
Let $h=f_2 f_3 \cdots f_{n-1}f_n: X=X^n \to X^1$. Then we have the following short exact sequence
\[ \sesshort T_{X/X^1}(-\log \DE_{n,n-1}).T_{X}(-\log \DE_{n,n-1}).h^*T_{X^1}(-\log \DE_{1,n}).. \]
Restricting to $X_s$ gives,
\[ \sesshort T_{X_s}(-\log (\DE_{n,n-1}|_{X_s})).T_{X}(-\log \DE_{n,n-1})|_{X_s}.h^*T_{X^1}|_{X_s}., \]
and restricting further to $X_y$ gives,
\[ \sesshort T_{X_s}|_{X_y}.T_{X}|_{X_y}.h^*T_{X^1}|_{X_y}.. \]
As in the smooth case (\ref{tower1}), we find that
\[ \sesshort h^*\OM_{X^1}|_{X_y}.\OM_X|_{X_y}.\OM_{X_s}|_{X_y}. \]
does not split. Thus, 
\[ \sesshort \sO_{X_y}.\OM_X^1(\log \DE_{n,n-1})|_{X_y} .\OM_{X_s}^1(\log (\DE_{n,n-1}|_{X_s}))|_{X_y}.\]
does not split. To show that $ \OM_X^1 (\log \DE_{n,n-1})|_{X_y}$ is ample, it suffices to show that $\OM_{X_s}^1(\log(\DE_{n,n-1}|_{X_s})|_{X_y}$ is ample.

We have
\[ \xymatrix@1{X_s = X^n_s \ar[r]^(.55){(f_n)_s} & X^{n-1}_s \ar[r]^{(f_{n-1})_s}  & \cdots  \cdots \ar[r] & X^3_s  \ar[r]^{(f_3)_s} & X^2_s}, \]
where each $X^i_s$ is a smooth projective variety over $\C$ of dimension $i-1$, and each $(f_i)_s:X^i_s \to X^{i-1}_s$ is a flat projective morphism with the property that $\var(f_i)_s = \dim(X^{i-1})_s$.  Furthermore,  $\om_{(X^i_s)_y}$ is ample for all $y \notin \DE_{i-1,i}|_{X^{i-1}_s}$.  

 For $2 \leq i \leq n-1$, define
 \[ B'_i,s := \{ p \in X^{i}_s - \DE_{i,i+1} |_{ X^{i}_s } |\: \al_{i,s}: \coh 0. X^{i+1}_p.(f_{i+1})_s^*T_{X^i_s}|_{X_p^{i+1}}. \to \coh 1.(X^{i+1}_s)_p.T_{(X^{i+1}_s)_p}.  \text{ is not injective}  \}. \]
As seen in (\ref{tower1})  $B'_{i,s} \subseteq B_i' \cap X^i_s$.  Set $B_{2,s} := B'_{2,s}$ and for $3 \leq i \leq n$ set
\[ B_{i,s} := B'_{i,s} + (f_i)^{-1}_s(B_{i-1,s}), \]
then by induction, $\OM_{X_s}^1(\log (\DE_{n,n-1}|_{X_s}))$ is quasi-ample with respect to 
\[X_s - (f_{n})^{-1}_s(B_{n-1,s})-\DE_{n,n-1}|_{X_s}.\]  
I claim that $(f_i)^{-1}_s(B_{i-1,s}) \subseteq (f_i)_s^{-1}(B_{i-1} \cap {X^{i-1}_s})$; indeed for $i=2$ this is true, since $B_{2,s} = B'_{2,s}$, and so 
 \begin{eqnarray*}
 (f_i)^{-1}_s(B_{i-1,s})  & =  & (f_i)^{-1}_s(B'_{i-1,s} + (f_{i-1})_s^{-1}(B_{i-2,s}))  \\
   & \subseteq  &  (f_i)^{-1}_s(B'_{i-1} \cap {X^{i-1}_s} + f_{i-1}^{-1}(B_{i-2}\cap {X^{i-2}_s})) \\
 & = & (f_i)_s^{-1}(B_{i-1}\cap{X^{i-1}_s}). 
 \end{eqnarray*}
Thus, since 
 \begin{eqnarray*}
\emptyset \neq X_y \cap  (X - \DE_{n,n-1} - f_n^{-1}(B_{n-1})) &=& X_y \cap (X_s - \DE_{n,n-1}|_{X_s} - f_n^{-1}(B_{n-1} \cap X_s^{i-1}))\\
&  \subseteq & X_y \cap (X_s  - (f_{n})^{-1}_s(B_{n-1,s})-\DE_{n,n-1}|_{X_s})
 \end{eqnarray*}
we have that $\OM_{X_s}^1(\log(\DE_{n,n-1}|_{X_s})|_{X_y}$ is ample.

Now suppose $\ga(C)$ is not in a fiber of $f_n$, so $f_n \ga:C \to X^{n-1}$ is non-constant.  By induction $\OM_{X^{n-1}}^1(\log \DE_{n-1,n-2})$ is quasi-ample with respect to $X^{n-1} - \DE_{n-1, n-2} - f_{n-1}^{-1}(B_{n-2}).$  Then since $\DE_{n-1} := (\DE_{n-1,n-2} + \DE_{n-1,n})_{\text{red}}$,   (\ref{2div}) implies that  $\OM_{X^{n-1}}^1(\log \DE_{n-1})$ is  quasi-ample with respect to $X^{n-1} - \DE_{n-1} - f_{n-1}^{-1}(B_{n-2}).$  

Therefore, we have the following short exact sequence
\[ \sesshort \ga^* f_n^* \OM_{X^{n-1}}^1(\log \DE_{n-1}). \ga^* \OM_X^1(\log \DE_{n,n-1}). \ga^* \OM_{X/X^{n-1}}^1(\log \DE_{n,n-1}). \]
with $\ga^* f_n^* \OM_{X^{n-1}}^1(\log \DE_{n-1})$ ample, and so it suffices to show $\ga^* \OM_{X/X^{n-1}}^1(\log \DE_{n,n-1})$ is ample.  

As in (\ref{notsmoothbase}), we have a smooth family $X - \DE_{n,n-1} - f_n^{-1}B_{n-1}  \to Y - \DE_{n-1} - B_{n-1}$ which satisfies the conditions of (\ref{pushforward}(2)), thus for $m$ sufficiently large and divisible $f_*\OM_{X/Y}^1(\log \DE_{n,n-1} + f^{-1}B)^m$ is ample with respect to $Y - \DE_{n-1} - B_{n-1}$.  Furthermore, for $m$ sufficiently large
\[ f^*f_*\OM_{X/X^{n-1}}^1(\log \DE_{n,n-1} + f^{-1}B_{n-1})^m \to \OM_{X/X^{n-1}}^1(\log \DE_{n-1} + f^{-1}B_{n-1})^m \]
is surjective over $X - \DE_{n-1,n} - f^{-1}B_{n-1}$.  Since $\ga(C) \cap(X - \DE_{n-1,n} - f^{-1}B_{n-1}) \neq \emptyset$, we find that $\ga^*\OM_{X/X^{n-1}}^1(\log \DE_{n-1} + f^{-1}B_{n-1})$ is an ample line bundle on $C$.  Since $f^{-1}B_{n-1}$ is reduced, we have that $\ga^*\OM_{X/X^{n-1}}^1(\log \DE_{n-1} + f^{-1}B_{n-1}) \iso \ga^*\OM_{X/X^{n-1}}^1(\log \DE_{n-1})$.
Thus  $\OM_{X^n}^1(\log \DE_{n,n-1})$ is quasi-ample with respect to $X^n - \DE_{n,n-1} - f_n^{-1}(B_{n-1})$.
\end{proof}

\section{Constructing Towers of Smooth Projective Varieties} \label{subsec:ex}

In this section we construct a tower of smooth projective varieties over the complex numbers and smooth morphisms between them of maximal variation.    
We first recall a construction of Kodaira, which for any $n$ produces a $g$ for which $\sM_g$ contains a complete $n$-dimensional subvariety, see \cite{HM}, \cite{O95}.

\begin{lemma} \textup{\cite{O95}} Let $K$ be a field, $D$ a curve over $K$ of genus $g(D) \geq 2$, and $Q \in D(K)$.  Suppose $\charact (K) \neq 2, 3$.  Then there exists a finite extension $K \subset L$ and a covering $C \to D$ of degree $3$ defined over $L$, totally ramified in $C \owns P \mapsto Q \in D$, and unramified elsewhere.
\end{lemma}

By keeping $D$ fixed, and letting the point $Q$ vary we construct a family of non-singular projective curves parametrized by a covering $D'$ of $D$ and its image is a complete curve in $\sM_{g(C)}$, where $g(C)=3g(D)-1$.  Iterating this construction, that is, considering all covers of degree three of curves in the family $\{ C_\la \}$ ramified at one point, we get a complete $2$-dimensional family of curves of genus $9g(D)-4$.  In general, we obtain a complete $n$-dimensional family of curves of genus $3^ng(D) - (3^n-1)/2$.

We now work over $\C$ and fix a curve $C_0$ of genus two.  Consider $\{C_\la\}$ the set of degree three covers of $C_0$ ramified in one point.  These are parametrized by some curve $B_0$, a cover of $C_0$.  Thus we get a map $g': \{ C_\la \} \to B_0$ with fibers $C_\la$ for $\la \in B_0$.  Now, $\{[C_\la]\}$, the image of $\{C_\la \}$ in $\sM_5$ is a complete curve, hence $g'$ must be nonisotrivial.

For each $\la \in B_0$, we can iterate this construction.  For a fixed $\la_0 \in B_0$, consider $\{C_{\la_0, \mu} \}$, the set of degree three covers of $C_{\la_0}$ ramified in one point, paramatrized by some curve $B_{\la_0}$.   For this fixed $\la_0$, $\{[C_{\la_0, \mu}] \}$ is a complete curve in $\sM_{14}$, so $f_{\la_0}:\{C_{\la_0, \mu} \} \to B_{\la_0}$ is nonisotrivial.  Letting $\la$ vary, we get a smooth projective morphism $f: X \to Y$ where $X=\{C_{\la, \mu}\}$ and $Y= \{ B_\la \}$ are projective varieties of dimension 3 and 2, respectively.  I claim that for any $p \in Y$, the set $\{ q \in Y | X_p \iso X_q \}$ is finite, so in particular $\var f = 2$.  Indeed, if $p,q \in B_{\la_0}$ for some $\la_0$, then since $f_{\la_0}$ is nonisotrivial, $X_p$ is not isomorphic to $X_q$.  Next suppose $F_1$ is any fiber of $f$ with $F_1 \in C_{\la_1, \mu}$, then, in particular,  $F_1$ is a covering of $C_{\la_1}$. But $F_1$ can only cover finitely many curves, hence $F_1$ can be isomorphic to at most finitely many other fibers $F_i \in C_{\la_i, \mu}$.  Thus any fiber of $f$ is isomorphic to finitely many other fibers, and hence for any  $p \in Y$, the set $\{ q \in Y | X_p \iso X_q \}$ is finite.

Let $g: Y \to B_0$ be the composition $Y= \{B_\la \} \to \{C_\la \} \to B_0$.  If $B_\la \iso B_{\la'}$ for general $\la, \la' \in B_0$, then either $C_\la \iso C_{\la'}$ for general $\la, \la' \in B_0$ or $B_\la$ covers infinitely many non-isomorphic curves, both of which lead to a contradiction.  Hence $g:Y \to B_0$ is nonisotrivial.  Thus we have the tower
\[ \xymatrix{X \ar[r]^f & Y \ar[r]^g & B_0 }\]
of smooth projective varieties with $\dim X = 3$, $\dim Y =2$, $\dim B_0=1$, and smooth morphisms of maximal variation.  One can iterate this construction to get a tower 
\[ \xymatrix{X^n \ar[r]^{f_n} & X^{n-1} \ar[r]^{f_{n-1}} & X^{n-2} \ar[r]^{f_{n-2}} & \cdots \cdots \ar[r]^{f_3} & X^2 \ar[r]^{f_2} & X^1} \]
where each $X^i$ is a smooth projective variety over $\C$ of dimension $i$, and for $2 \leq i \leq n $, $f_i:X^i \to X^{i-1}$ is a smooth, projective morphism satisfying the property that for all $y \in X^{i-1}$, the set $\{p \in Y | X^i_p \iso X^i_y \}$ is finite, so in particular $\var(f_i) = \dim X^{i-1}$.

\nocite{*}
\bibliographystyle{plain}
\bibliography{positivity}

\end{document}